\documentclass[11pt]{amsart}
\usepackage{calc,amssymb,amsthm,amsmath}
\RequirePackage[dvipsnames,usenames]{xcolor}
\usepackage{hyperref}
\hypersetup{
bookmarks,
bookmarksdepth=3,
bookmarksopen,
bookmarksnumbered,
pdfstartview=FitH,
colorlinks,backref,hyperindex,
linkcolor=Sepia,
anchorcolor=BurntOrange,
citecolor=MidnightBlue,
citecolor=OliveGreen,
filecolor=BlueViolet,
menucolor=Yellow,
urlcolor=OliveGreen
}
\usepackage{alltt}
\usepackage{multicol}
\usepackage{etex}
\usepackage{xspace}
\usepackage{stmaryrd}
\usepackage{tikz}

\input{kmacros3.sty}

\usepackage[left=1.2in,top=1.0in,right=1.2in,bottom=1.0in]{geometry}

\usepackage{verbatim}
\usepackage{enumerate}

\usepackage{marginnote}

\DeclareMathOperator{\hgt}{ht}

\DeclareMathOperator{\fm}{\mathfrak{m}}

\DeclareMathOperator{\fp}{\mathfrak{p}}
\DeclareMathOperator{\fq}{\mathfrak{q}}
\DeclareMathOperator{\EE}{E}
\DeclareMathOperator{\HH}{H}
\DeclareMathOperator{\II}{I}
\usepackage{graphicx}

\usepackage[all,cmtip]{xy}

\DeclareMathOperator{\cF}{\mathcal{F}}
\DeclareMathOperator{\cM}{\mathcal{M}}

\DeclareMathOperator{\image}{Im}
\DeclareMathOperator{\Ass}{Ass}
\DeclareMathOperator{\injdim}{inj.dim}

\DeclareMathOperator{\End}{{End}}

\renewcommand{\phi}{\varphi}

\renewcommand{\theta}{\vartheta}

\renewcommand{\epsilon}{\varepsilon}

\renewcommand{\to}[1][]{\xrightarrow{\ #1\ }}

\title{On injective dimension of $F$-finite $F$-modules and holonomic $D$-modules}

\author{Wenliang Zhang}
\address{Department of Mathematics, Statistics, and Computer Science, University of Illinois, Chicago, IL 60607, USA}
\email{wlzhang@uic.edu}
\thanks{The author is partially supported by the National Science Foundation.}

\begin{document}
\maketitle

\begin{abstract}
We investigate injective dimension of $F$-finite $F$-modules in characteristic $p$ and holonomic $D$-modules in characteristic 0. One of our main results is the following. If either
\begin{enumerate}
\item $R$ is a regular ring of finite type over an infinite field of characteristic $p>0$ and $\scr{M}$ is an $F_R$-finite $F_R$-module; or
\item $R=k[x_1,\dots,x_n]$ where $k$ is a field of characteristic 0 and $\scr{M}$ is a holonomic $D(R,k)$-module. 
\end{enumerate}
then $\injdim_R(\scr{M})=\dim(\Supp_R(\scr{M}))$.
\end{abstract}
\section{Introduction}
Let $R$ be a regular commutative noetherian of characteristic $p$ and let $\injdim_R(M)$ denote the injective dimension of an $R$-module $M$. It was proved in \cite{HunekeSharpBassnumbersoflocalcohomologymodules} that $\injdim_R(\HH^i_J(R))\leq \dim(\Supp_R(\HH^i_J(R)))$ for each ideal $J$ of R, where $\HH^i_J(R)$ denotes the $i$th local cohomology of $R$ supported in an ideal $J$. This result was then generalized further in \cite{LyubeznikFModulesApplicationsToLocalCohomology} which introduced a theory of $F_R$-modules (this will be reviewed in Section \ref{section: preliminaries}) and proved that $\injdim_R(\scr{M})\leq \dim(\Supp_R(\scr{M}))$ for each $F_R$-module $\scr{M}$ and that $\HH^i_J(R)$ is an $F_R$-module. 

In an interesting paper \cite{PuthenpurakaInjectiveResolutionofLC}, it is proved that $\injdim_R(\mathcal{T}(R))=\dim(\Supp_R(\mathcal{T}(R)))$ for a polynomial ring $R=k[x_1,\dots,x_n]$ in characteristic 0. Here $\mathcal{T}$ is the Lyubeznik functor. Due to its technicality we omit the definition of $\mathcal{T}$ and refer the reader to \cite{LyubeznikFinitenessLocalCohomology} for details. We should remark that a primary example of $\mathcal{T}$ is the repeated local cohomology functor $\HH^{i_1}_{j_1}\cdots \HH^{i_s}_{J_s}(-)$ and that $\mathcal{T}(R)$ is a holonomic $D(R,k)$-module (theory of $D(R,k)$-modules will be reviewed in Section \ref{section: preliminaries}). It's asked in \cite[page 711]{PuthenpurakaInjectiveResolutionofLC} whether the same result holds in characteristic $p$. The main goal of this short note is twofold: to give a positive answer to this question in characteristic $p$ and to prove a stronger result in characteristic 0. Here are our main results.

\begin{theorem}[Theorems \ref{theorem: only max ideal at last spot} and \ref{theorem: only max ideal at last spot for D-modules}]
\label{thm: intro last spot}
Assume either 
\begin{enumerate}
\item $R$ is a commutative noetherian regular Jacobson ring of characteristic $p>0$ and $\scr{M}$ is an $F_R$-finite $F_R$-module; or
\item $R=k[x_1,\dots,x_n]$ is a polynomial over a field $k$ of characteristic 0 and $\scr{M}$ is a holonomic $D(R,k)$-module.
\end{enumerate} 
Set $t:=\injdim_R(\scr{M})$. Then \[\mu^t(\fp,\scr{M})=0\] for each non-maximal prime ideal $\fp$ of $R$, where $\mu^t(\fp,\scr{M})$ is the $t$-th Bass number of $\scr{M}$ with respect to $\fp$ ({\it i.e.} $\mu^t(\fp,\scr{M})=\dim_{\kappa(\fp)}\Ext^t_{R_{\fp}}(\kappa(\fp),\scr{M}_{\fp})$ ). 
\end{theorem}

\begin{theorem}[Theorems \ref{theorem: injective dim of F-module in char p} and \ref{thm: injective dim of holonomic D-module}]
\label{theorem: injective dimension}
Assume either
\begin{enumerate}
\item $R$ is a regular ring of finite type over an infinite field $k$ of characteristic $p>0$ and $\scr{M}$ is an $F_R$-finite $F_R$-module; or
\item $R=k[x_1,\dots,x_n]$ is a polynomial over a field $k$ of characteristic 0 and $\scr{M}$ is a holonomic $D(R,k)$-module.
\end{enumerate} 
Then 
\[\injdim_R(\scr{M})=\dim_R(\Supp_R(\scr{M})).\]

\end{theorem}

\subsection*{Acknowledgements} The author would like to thank Gennady Lyubeznik for helpful discussions, David Ben-Zvi and Daniel Caro for answering questions on $D$-modules, and Tony Puthenpurakal for comments on a draft of this paper. The author is grateful to the referee for his/her suggestions that improve the exposition of this paper.
\section{Preparatory results on $F_R$-modules and $D$-modules}
\label{section: preliminaries}
In this section, we review some basic notions and results in the theories of $F$-modules, $D$-modules and Jacobson rings. We also prove some new results on $F$-modules and $D$-modules that are needed in the sequel.

\subsection{$F_R$-modules} Let $R$ be a commutative noetherian regular ring of characteristic $p$. Let $F_R$ denote the Peskine-Szpiro functor: 
\[F_R(M):=R^{(1)}\otimes_RM\]
for each $R$-module $M$, where $R^{(1)}$ denote the $R$-module that is the same as $R$ as a left $R$-module and whose right $R$-module structure is given by $r'\cdot r=r^pr'$ for all $r'\in R^{(1)}$ and $r\in R$.

\begin{remark}
\label{remark: F commutes with homomorphism}
Given a homomorphism $\phi:R\to R'$ of rings of characteristic $p$, it is clear that $\phi(r^p)=\phi(r)^p$ for each $r\in R$, Hence $\phi\circ F_R=F_S\circ \phi$. Consequently there is an identification of functors $R' \otimes_R F_R(-)=F_{R'}(R'\otimes_R-)$, {\it i.e.} $R' \otimes_R F_R(M)=F_{R'}(R'\otimes_RM)$ for each $R$-module $M$ and it is functorial in $M$.

In particular, if $S$ is a multiplicatively closed subset of $R$, we have $S^{-1}F_R(M)=F_{S^{-1}R}(S^{-1}M)$.
\end{remark}

\begin{definition}[Definitions 1.1, 1.9 and 2.1 in \cite{LyubeznikFModulesApplicationsToLocalCohomology}]
\label{definition: F-modules}
An $F_R$-module is an $R$-module $\scr{M}$ equipped with an $R$-linear isomorphism $\theta_{\scr{M}}:\scr{M}\to F_R(\scr{M})$. 

A homomorphism between $F_R$-modules $(\scr{M},\theta_{\scr{M}})$ and $(\scr{N},\theta_{\scr{N}})$ is a homomorphism $\varphi:\scr{M}\to \scr{N}$ such that the following is a commutative diagram
\[\xymatrix{
\scr{M} \ar[r]^{\theta_{\scr{M}}} \ar[d]^{\varphi} & F_R(\scr{M}) \ar[d]^{F_R(\varphi)}\\
\scr{N} \ar[r]^{\theta_{\scr{N}}} & F_R(\scr{N}).
}\]

A generating morphism of an $F_R$-module $(\scr{M},\theta_{\scr{M}})$ is an $R$-linear map $\beta:M\to F_R(M)$ of an $R$-module $M$ such that the direct limit of the following diagram is the same as $\theta_{\scr{M}}:\scr{M}\to F_R(\scr{M})$.
\[
\xymatrix{
M \ar[r] \ar[d] & F_R(M) \ar[r]^{F_R(\beta)} \ar[d]^{F^2_R(\beta)} & F^2_R(M)\ar[r] \ar[d]^{F^3_R(\beta)} &\cdots \\
F_R(M) \ar[r]^{F_R(\beta)} & F^2_R(M) \ar[r]^{F^2_R(\beta)} & F^3_R(M)\ar[r] &\cdots 
}\] 

An $F_R$-module $\scr{M}$ is called $F_R$-finite if it admits a generating homomorphism $\beta:M\to F_R(M)$ such that $M$ is a finitely generated $R$-module. 
\end{definition}

We collect some basic results on $F_R$-modules as follows.
\begin{remark}
\label{proposition: properties of F-modules}
Let $R$ be a commutative noetherian regular ring of characteristic $p$.
\begin{enumerate}
\item $R$ and $R_f$ are $F_R$-finite $F_R$-modules for each element $f\in R$, and the natural map $R\to R_f$ is an $F$-module homomorphism (\cite[Example 1.2]{LyubeznikFModulesApplicationsToLocalCohomology}).
\item Every injective $R$-module is an $F_R$-module (\cite[Proposition 1.5]{HunekeSharpBassnumbersoflocalcohomologymodules}).
\item A minimal injective resolution of an $F_R$-module is also a complex of $F$-modules and $F$-module homomorphisms (\cite[Example 1.2(b")]{LyubeznikFModulesApplicationsToLocalCohomology})
\item All $F_R$-finite $F_R$-modules form an abelian subcategory of the category of $R$-modules (\cite[Theorem 2.8]{LyubeznikFModulesApplicationsToLocalCohomology}). Hence each local cohomology module $\HH^i_J(R)$ is an $F_R$-finite $F_R$-module for all ideals $J$ of $R$ and all $i\geq 0$.
\item Let $S$ be a multiplicatively closed subset of $R$. It follows from Remark \ref{remark: F commutes with homomorphism} that $S^{-1}\scr{M}$ is an $F_{S^{-1}R}$-module for each $F_R$-module $\scr{M}$. Moreover, if $\scr{M}$ is $F_R$-finite, then $S^{-1}\scr{M}$ is $F_{S^{-1}R}$-finite.
\item If $\scr{M}$ is a simple $F_R$-module, then $S^{-1}\scr{M}$ is either 0 or a simple $F_{S^{-1}R}$-module. Consequently if $\scr{M}$ has finite length in the category of $F_R$-module, then $S^{-1}\scr{M}$ will have finite length in the category of $F_{S^{-1}R}$-modules.
\end{enumerate}
\end{remark}

\begin{remark}
\label{remark: finite length F-module finitely associated primes}
Let $R$ be a commutative noetherian regular ring of characteristic $p$. If an $F_R$-module $M$ has finite length in the category of $F_R$-modules, then $M$ has only finitely many associated primes. To see this, consider a composition series of $M$ with (finitely many) factors $M_i$ which are simple $F_R$-modules. Note that $\Ass_R(M)\subseteq \cup_i\Ass_R(M_i)$ and any simple $F_R$-module has only one associated prime (\cite[Theorem 2.12(b)]{LyubeznikFModulesApplicationsToLocalCohomology}). It follows that $M$ has only finitely many associated primes.
\end{remark}

\subsection{$D$-modules.} Let $C$ be a commutative ring. \emph{Differential operators} on $C$ are defined inductively as follows: for each $r\in C$, the multiplication by $r$ map $\tilde{r}\colon C\to C$ is a differential operator of order $0$; for each positive integer $n$, the differential operators of order less than or equal to $n$ are those additive maps $\delta\colon C\to C$ for which the commutator
\[
[\tilde{r},\delta]\ =\ \tilde{r}\circ\delta-\delta\circ\tilde{r}
\]
is a differential operator of order less than or equal to $n-1$. If $\delta$ and $\delta'$ are differential operators of order at most $m$ and $n$ respectively, then $\delta\circ\delta'$ is a differential operator of order at most $m+n$. Thus, the differential operators on~$R$ form a subring $D(C)$ of $\End_{\ZZ}(C)$.

When $C$ is an algebra over a commutative ring $A$, we define $D(C,A)$ to be the subring of~$D(C)$ consisting of differential operators that are $A$-linear. 

We believe that the following proposition is well-known; we include a proof since we couldn't find a proper reference.
\begin{proposition}
\label{proposition: localization preserve finite length}
Assume that $C$ is an integral domain and let $S$ be a multiplicatively closed subset of $C$. If $M$ is a simple $D(C,A)$-module, then $S^{-1}M$ is also a simple $D(S^{-1}C,A)$-module.

Consequently, if a $D(C,A)$-module $M$ has finite length in the category of $D(C,A)$-module, then $S^{-1}M$ also has finite length in the category of $D(S^{-1}C,A)$modules.
\end{proposition}
\begin{proof}
Note that $M$ is a simple $D(C,A)$-module if and only if $D(C,A)z=M$ for each nonzero element $z\in M$. If $S^{-1}M=0$, our conclusion is clear. Assume that $S^{-1}M\neq 0$. Each $A$-linear differential operator on $C$ acts naturally and $A$-linearly on $S^{-1}C$ via the quotient rule, hence we may view $D(C,A)$ as a subset of $D(S^{-1}C,A)$, it follows that $D(S^{-1}C,A)y=S^{-1}M$ for each nonzero element $y\in S^{-1}M$. Hence $S^{-1}M$ is also a simple $D(S^{-1}C,A)$-module.

The second part of our proposition follows from considering a composition series of $M$ the category of $D(C,A)$-modules.
\end{proof}


\begin{remark}
\label{remark: finite length implies finite associated primes}
Assume that $C$ is noetherian. If a $D(C,A)$-module $M$ has finite length in the category of $D(C,A)$-modules, then it has only finitely many associated primes as a $C$-module. To see this, note that it suffices to prove this for a simple $D(C,A)$-module. Assume that $M$ is a simple $D(C,A)$-module. Let $\fp$ be a maximal member among all its associated primes. Then $\HH^0_{\fp}(M)$ is a nonzero $D(C,A)$-submodule of $M$, so $\HH^0_{\fp}(M)=M$ since $M$ is simple. This shows that $M$ has only one associated prime and finishes the proof.
\end{remark}

\begin{proposition}
\label{prop: D-module localized at minimal prime}
Let $R=k[x_1,\dots,x_n]$ be a polynomial ring over a field $k$ and $M$ be a $D(R,k)$-module. Assume that $\fp\subset R$ is a minimal prime of $M$. Then $M_{\fp}$ is an injective $R_{\fp}$-module.
\end{proposition}
\begin{proof}
\cite[page 211]{LyubeznikCharFreeInjDim} proves the case when $R=k[[x_1,\dots,x_n]]$, but the same proof works for polynomial rings as well. 
\end{proof}

\begin{proposition}
\label{prop: injdim for D-module}
Let $R=k[x_1,\dots,x_n]$ be a polynomial ring over a field $k$ and $M$ be a $D(R,k)$-module. Then $\injlim_R(M)\leq \dim_R(\Supp_R(M))$.
\end{proposition}
\begin{proof}
\cite[Theorem 1]{LyubeznikCharFreeInjDim} proves the case when $R=k[[x_1,\dots,x_n]]$, but the same proof works for polynomial rings as well. 
\end{proof}

Next we would like to recall the notion of a holonomic $D$-module that will be used in the sequel; our main reference is the book \cite{BjorkBookRIngDiffOperators}. 

Let $R=k[x_1,\dots,x_n]$ be a polynomial ring over a field $k$ of characteristic 0. Then it is well-known that $D(R,k)=R\langle \partial_1,\dots,\partial_n\rangle$ where $\partial_i=\frac{\partial}{\partial x_i}$. Set $\cF_i$ to be the $k$-linear span of the following set
\[\{x^{a_1}_1\cdots x^{a_n}_n\partial^{b_1}_1\cdots \partial^{b_n}_n\mid \sum_{j=1}^n a_j+\sum_{j=1}^n b_j\leq i\}.\]
Then $\cF_0\subseteq \cF_1\subseteq \cdots$ is a filtration of $D(R,k)$, called the Bernstein filtration. It is well-known that the graded ring $gr^{\cF}(D(R,k))$ associated with the Bernstein filtration is isomorphic to $k[x_1,\dots,x_n,\xi_1,\dots,\xi_n]$ where $\xi_j$ denotes the image of $\partial_j$ in $gr^{\cF}(D(R,k))$. If $M$ is a finitely generated $D(R,k)$-module, then $M$ admits a filtration of finite dimensional $k$-spaces $\cM_0\subseteq \cM_1\subseteq \cdots$ with the properties that $\cup_i\cM_i=M$ and $\cF_i\cM_j\subseteq \cM_{i+j}$. Then the graded module $gr^{\cM}(M)$ associated to the filtration $\cM$ is naturally a finitely generated $gr^{\cF}(D(R,k))$-module. A finitely generated $D(R,k)$-module $M$ is called {\it holonomic} if it is either 0 or the dimension of $gr^{\cM}(M)$ over $gr^{\cF}(D(R,k))$ is $n$.

\begin{remark}
\label{rmk: growth defn of holonomicity}
A $k$-{\it filtration} on a $D(R,k)$-module $M$ is an ascending chain of finite-dimensional $k$-vector spaces $\cM_0\subset \cM_1\subseteq \cdots$ such that $\cup_i\cM_i=M$ and $\cF_i\cM_j\subset \cM_{i+j}$ for all $i$ and $j$. It is proved in \cite{Bavulafiltratioholonomic} and\cite{LyubeznikCharFreeHolonomic} that $M$ is holonomic if and only if there is a constant $\eta$ such that $\dim_k(\cM_i)\leq \eta i^n$ for all $i$. 
\end{remark}


\begin{proposition}
\label{prop: holonomic localized}
Let $R=k[x_1,\dots,x_n]$ be a polynomial ring over a field $k$ of characteristic 0 and $M$ be a holonomic $D(R,k)$-module. Let $S=k[x_n]\backslash\{0\}$ and $R'=S^{-1}R=k(x_n)[x_1,\dots,x_{n-1}]$. Then $S^{-1}M$ is also a holonomic $D(R',k(x_n))$-module.
\end{proposition}
\begin{proof}
Since $M$ is holonomic, it is cyclic (\cite[Corollary 8.19 in Chapter 1]{BjorkBookRIngDiffOperators}). Assume that $M$ is generated by $z$. Set $\cM_i=\cF_i\cdot z$. Then $\{\cM_i\}_i$ is a filtration on $M$ with the properties that $\cup_i\cM_i=M$ and $\cF_i\cM_j\subseteq \cM_{i+j}$. Let $A=\frac{gr^{\cF}(D(R,k))}{\Ann_{gr^{\cF}(D(R,k))}(gr^{\cM}(M))}$. Then $\dim(A)=n$ since $M$ is holonomic. Let $\bar{x}_i$ and $\bar{\xi}_j$ denote the images of $x_i$ and $\xi_j$ in $A$ for $i,j=1,\dots,n$. By Noether Normalization (\cite[Exercise 16 on page 69]{AtiyahMacdonaldBook}), after a linear change of variables, we may assume that $x_1,\dots, x_n,\xi_1,\dots,\xi_n$ can be arranged into $x_{i_1},\dots,x_{i_n},\xi_{j_1},\dots,\xi_{j_n}$ such that $A$ is a finitely generated $A'=k[\bar{x}_{i_1},\dots,\bar{x}_{i_t},\bar{\xi}_{j_1},\dots,\bar{\xi}_{j_{n-t}}]$-module and $\bar{x}_{i_{t+1}},\dots,\bar{x}_{i_n},\bar{\xi}_{j_{n-t+1}},\dots,\bar{\xi}_{j_n}$ are integral over $A'$, and $x_{i_1}=x_n$. Let $N$ be the maximum of the degrees of the monic polynomials associated with integral dependence of $\bar{x}_{i_{t+1}},\dots,\bar{x}_{i_n},\bar{\xi}_{j_{n-t+1}},\dots,\bar{\xi}_{j_n}$ over $A'$. Then $\cF_i\cdot z$ is the same as the $k$-linear span of the following set 
\[\Big\{x^{a_1}_{i_1}\cdots x^{a_t}_{i_t}x^{a_{t+1}}_{i_{t+1}}\cdots x^{a_{n}}_{i_{n}}\partial^{b_1}_{j_1}\cdots \partial^{b_{n-t}}_{j_{n-t}}\partial^{b_{n-t+1}}_{j_{n-t+1}}\cdots \partial^{b_n}_{j_n} \cdot z\mid \substack{\sum_{j=1}^n a_j+\sum_{j=1}^n b_j\leq i \\ a_{t+1},\dots,a_n,b_{n-t+1},\dots,b_n\leq N}\Big\}\]
Therefore, $S^{-1}\cF_i\cdot z$ is the same as the $k(x_n)$-span of the following set
\[\Big\{x^{a_2}_{i_2}\cdots x^{a_t}_{i_t}x^{a_{t+1}}_{i_{t+1}}\cdots x^{a_{n}}_{i_{n}}\partial^{b_1}_{j_1}\cdots \partial^{b_{n-t}}_{j_{n-t}}\partial^{b_{n-t+1}}_{j_{n-t+1}}\cdots \partial^{b_n}_{j_n} \cdot z\\ \mid \substack{\sum_{j=1}^n a_j+\sum_{j=1}^n b_j\leq i \\ a_{t+1},\dots,a_n,b_{n-t+1},\dots,b_n\leq N}\Big\}\]
which produces a $k(x_n)$-filtration of $S^{-1}M$ as a $D(R',k(x_n))$-module. It is clear that there is a constant $\eta$ such that $\dim_{k(x_n)}(\cF'_i\cdot z)\leq \eta i^{n-1}$ for all $i$. Hence $S^{-1}M$ is a holonomic $D(R',k(x_n))$-module by Remark \ref{rmk: growth defn of holonomicity}.
\end{proof}

We end this section by collecting some basic results on Jacobson rings.
\subsection{Jacobson rings} A commutative ring $R$ is called a {\it Jacobson ring} (or a {\it Hilbert ring}) if every prime ideal is the intersection of all maximal ideals that contain it. We will collect some well-known facts about Jacobson rings and our main reference is \cite[\S10]{EGA4.3}.
\begin{proposition}
\label{proposition: properties of Jacobson rings}
Let $R$ be a Jacobson noetherian ring. 
\begin{enumerate}
\item Let $R$ be a Jacobson noetherian ring. Then $R$ has only finitely many maximal ideals if and only if $\dim(R)=0$.
\item Any homomorphic image of a Jacobson ring is also a Jacobson ring.
\item Let $R$ be a Jacobson noetherian ring. Then the localization $R_f$ is a Jacobson ring for each element $f\in R$ and there is a one-to-one correspondence between the maximal ideal of $R_f$ and the maximal ideals of $R$ that don't contain $f$.   
\item Any finitely generated algebra over an infinite field is a Jacobson ring.
\end{enumerate}
\end{proposition}

\begin{remark}
\label{remark: existence of maximal ideal}
One consequence of Proposition \ref{proposition: properties of Jacobson rings} is that, given any finitely many prime ideals $\fp_1,\dots,\fp_m$ in a Jacobson ring, there exists a maximal ideal that does not contain any of $\fp_1,\dots,\fp_m$.
\end{remark}

\section{Injective dimension of $F_R$-finite $F_R$-modules}
In this section, we study the injective dimension of an $F_R$-finite $F_R$-module. To this end, we begin with an analysis of $F_R$-finiteness of $\EE(R/\fp)$ where $R$ is a commutative noetherian regular ring of characteristic $p$. Recall that $\EE(R/\fp)$ is always an $F_R$-module by Remark \ref{proposition: properties of F-modules}. 

The next two propositions are applications of the celebrated result that any $F_R$-finite $F_R$-module has only finitely many associated primes \cite[Theorem 2.12(a)]{LyubeznikFModulesApplicationsToLocalCohomology}. 


\begin{proposition}
\label{proposition: height d-1}
Let $R$ be a commutative noetherian regular ring containing a field of characteristic $p>0$. Let $d=\dim(R)$ and $\fp$ be a prime ideal of height $d-1$. Then $\EE(R/\fp)$ is $F_R$-finite if and only if $\fp$ is contained in finitely many maximal ideals. 

In particular, if $R$ is also a Jacobson ring of positive dimension, then $\EE(R/\fp)$ is not $F_R$-finite.
\end{proposition}
\begin{proof}
Set $\II^j=\bigoplus_{\hgt(\fq)=j}\EE(R/\fq)$, where the direct sum is taken over all height $j$ prime ideals $\fq$. Since $R$ is regular and hence Gorenstein, $0\to R\to \II^0\to \cdots\to \II^j\to \cdots \to \II^d\to 0$ is a minimal injective resolution of $R$. Since the height of $\fp$ is $d-1$, according to Hartshorne-Lichtenbaum Vanishing Theorem \cite[8.2.1]{LCBookBrodmannSharp}, we have an exact sequence
\begin{equation}
\label{ses height d-1}
0\to \HH^{d-1}_{\fp}(R)\to \EE(R/\fp)\to \bigoplus_{\fp\subset \fm;\ \hgt(\fm)=d}\EE(R/\fm)\to 0
\end{equation}
This is also an exact sequence in the category of $F_R$-modules (Remark \ref{proposition: properties of F-modules}(c)). 

If $\EE(R/\fp)$ is $F_R$-finite, then so will be $\bigoplus_{\fp\subset \fm;\ \hgt(\fm)=d}\EE(R/\fm)$ since $F_R$-finite $F_R$-modules form an abelian category \cite[Theorem 2.8]{LyubeznikFModulesApplicationsToLocalCohomology}. Consequently, $\bigoplus_{\fp\subset \fm;\ \hgt(\fm)=d}\EE(R/\fm)$ must have finitely many associated primes by \cite[Theorem 2.12(a)]{LyubeznikFModulesApplicationsToLocalCohomology}. It is clear that the associated primes of $\bigoplus_{\fp\subset \fm;\ \hgt(\fm)=d}\EE(R/\fm)$ are precisely the maximal ideals containing $\fp$. Hence $\fp$ is contained in finitely many maximal ideals.

On the other hand, if $\fp$ is contained in finitely many maximal ideals. Then $\bigoplus_{\fp\subset \fm;\ \hgt(\fm)=d}\EE(R/\fm)$ is a direct sum of finitely many $F_R$-finite $F_R$-module and is $F_R$-finite. It follows from (\ref{ses height d-1}) that $\EE(R/\fp)$ is an extension of two $F_R$-finite $F_R$-modules, hence it is $F_R$-finite.
\end{proof}


As we will see next, once the height of a prime ideal $\fp$ is $\leq d-2$, then $\EE(R/\fp)$ is never $F_R$-finite, no matter how many maximal ideals contain $\fp$.

\begin{proposition}
\label{proposition: height d-2}
Let $R$ be a commutative noetherian regular ring containing a field of characteristic $p>0$. Let $d=\dim(R)$ and $\fp$ be a prime ideal of height $\leq d-2$. Then $\EE(R/\fp)=\EE(R/\fp)_{\fm}$ is not $F_{R_{\fm}}$-finite $F_{R_{\fm}}$-module for each maximal ideal $\fm$ that contains $\fp$.

In particular, if $\hgt(\fp)\leq d-2$, then $\EE(R/\fp)$ is not $F_R$-finite.
\end{proposition} 
\begin{proof}
First, we prove the case when $\hgt(\fp)=d-2$ and we will follow the same strategy as in the proof of Proposition \ref{proposition: height d-1}. Note that if $M$ is $F_R$-finite (or has finite length in the category of $F_R$-modules), then $M_{\fm}$ will be $F_{R_{\fm}}$-finite (or will have finite length in the category of $F_{R_{\fm}}$-modules). Replacing $R$ by $R_{\fm}$, we may assume that $R$ is now a regular local ring. Set $\II^j=\bigoplus_{\hgt(\fq)=j}\EE(R/\fq)$, where the direct sum is taken over all height $j$ prime ideals $\fq$. Then $0\to R\to\II^0\to \cdots\to \II^j\xrightarrow{\delta^j} \cdots \to \II^d=\EE(R/\fm)\to 0$ is a minimal injective resolution of $R$. Since $\hgt(\fp)=d-2$, applying $\Gamma_{\fp}$ to this injective resolution of $R$ produces 3 short exact sequences:
\begin{align}
0 &\to \HH^{d-2}_{\fp}(R) \to \EE(R/\fp) \to \image(\delta^{d-2}) \to 0\tag{a} \\
0& \to  \image(\delta^{d-2}) \to \ker(\delta^{d-1}) \to \HH^{d-1}_{\fp}(R) \to 0\tag{b} \\
0&\to \ker(\delta^{d-1}) \to \II^{d-1} \to \II^d=\EE(R/\fm)\to 0 \tag{c}
\end{align}
where (c) follows from Hartshorne-Lichtenbaum Vanishing Theorem. If $\EE(R/\fp)$ were $F_R$-finite (or had finite length in the category of $F_R$-modules), then by (a) $\image(\delta^{d-2})$ would also be $F_R$-finite (or would have finite length in the category of $F_R$-modules). Then (b) would imply that $\ker(\delta^{d-1})$ would be $F_R$-finite (or have finite length) since $\HH^{d-1}_{\fp}(R)$ is $F_R$-finite (or has finite length). Then (c) would imply that $\II^{d-1} $ would be $F_R$-finite (or have finite length). Consequently by \cite[Theorem 2.12(a)]{LyubeznikFModulesApplicationsToLocalCohomology} (or by Remark \ref{remark: finite length F-module finitely associated primes}) $\II^{d-1} $ would have only finitely many associated primes. But this is not the case; there are infinitely many height $d-1$ primes that contain $\fp$ and each of them is an associated prime of $\II^{d-1}$. This proves the case when $\hgt(\fp)=d-2$.

Next, assume that $\hgt(\fp)\leq d-3$. Let $\fq$ be a prime ideal of height $\hgt(\fp)+2$ and containing $\fp$. Then the height of $\fp R_{\fq}$ is exactly 2 less than the dimension of $R_{\fq}$; hence by our previous paragraph we know that $\EE(R/\fp)=\EE(R/\fp)_{\fq}=\EE(R_{\fq}/\fp R_{\fq})$ is {\it not} $F_{R_{\fq}}$-finite. Thus, $\EE(R/\fp)$ is {\it not} $F_R$-finite.
\end{proof}





\begin{theorem}
\label{theorem: only max ideal at last spot}
Let $R$ be a $d$-dimensional commutative noetherian regular Jacobson ring of characteristic $p>0$. Assume that $\scr{M}$ is an $F_R$-finite $F_R$-module. Set $\injdim_R(\scr{M})=t$. Then $\mu^t(\fp,\scr{M})=0$ for each non-maximal prime ideal $\fp$. 
\end{theorem}
\begin{proof}
According to \cite[Lemma 1.4]{LyubeznikFinitenessLocalCohomology}, $\mu^t(\fp,\scr{M})=\mu^0(\fp,\HH^t_{\fp}(\scr{M}))$. Assume that $\mu^0(\fp,\HH^t_{\fp}(\scr{M}))\neq 0$ and we will look for a contradiction. 

Since $\mu^0(\fp,\HH^t_{\fp}(\scr{M}))\neq 0$, we must have $\HH^t_{\fp}(\scr{M})_{\fp}\neq 0$; consequently, $\fp$ (being the unique minimal element in the support of $\HH^t_{\fp}(\scr{M})$) must be an associated prime of $\HH^t_{\fp}(\scr{M})$. Under our assumption on $\scr{M}$, we have that $\HH^t_{\fp}(\scr{M})$ has only finitely many associated primes. \vspace{5mm}

{\it Claim.} $\Ass_R(\HH^t_{\fp}(\scr{M}))=\{\fp\}$. 
\begin{proof}[Proof of Claim]
Assume otherwise and let $\fp,\fq_1,\dots,\fq_m$ be the associated primes of $H^t_{\fp}(\scr{M})$, and set $J=\fq_1\cdots \fq_m$. Then $L:=\HH^0_J(\HH^t_{\fp}(\scr{M}))$ is also $F$-finite. Let $N=\HH^t_{\fp}(\scr{M})/L$. We will show that $N=0$, which will produce a contradiction since $\fp$ is {\it not} an associated prime of $L$.

Since $t=\injdim(\scr{M})$, it follows that $\HH^t_{\fp}(\scr{M})$ is a quotient of an injective $R$-module. Given any element $f\in R$, the multiplication by $f$ on any injective module is surjective, hence it also surjective on $N$ and any localization of $N$.

If $\hgt(\fp)=d-1$, then $\fq_1,\dots,\fq_m$ are maximal ideals. Hence $L$ is an injective $R$-module, hence $\HH^t_{\fp}(\scr{M})=L\oplus N$. Since $N$ is a submodule of $\HH^t_{\fp}(\scr{M})$, each associated prime must be an associated prime of $\HH^t_{\fp}(\scr{M})$. It is clear that none of $\fq_1,\dots,\fq_m$ is an associated prime of $N$. Therefore $\fp$ is the only associated prime of $N$. Consequently multiplication by $f\notin \fp$ is also injective on $N$. Thus, $N=N_{\fp}$. Since $N_{\fp}$ is an $F_{R_{\fp}}$-finite $F_{R_{\fp}}$-module and $\dim_{R_{\fp}}(N_{\fp})=0$, it follows from \cite[Theorem 1.4]{LyubeznikFModulesApplicationsToLocalCohomology} that $N_{\fp}$ is a direct sum of finitely copies of $\EE(R_{\fp}/\fp R_{\fp})=\EE(R/\fp)$. To summarize, we have shown that $N$, which is $F_R$-finite, is a direct sum of finitely many copies of $\EE(R/\fp)$. Since $R$ is a Jacobson ring, so is $R/\fp$ (Proposition \ref{proposition: properties of Jacobson rings}). Hence there are infinitely many maximal ideals that contain $\fp$. By Proposition \ref{proposition: height d-1}, $\EE(R/\fp)$ is {\it not} $F_R$-finite; thus $N$ must be 0.

Assume now $\hgt(\fp)\leq d-2$. Since $R$ is a Jacobson ring, there exists a maximal ideal $\fm$ that contains $\fp$ but not any of $\fq_1,\dots,\fq_m$ (Remark \ref{remark: existence of maximal ideal}). Hence $N_{\fm}=\HH^t_{\fp}(\scr{M})_{\fm}$. Over $R_{\fm}$, the only associated prime of $\HH^t_{\fp}(\scr{M})_{\fm}=N_{\fm}$ is $\fp R_{\fm}$. Consequently multiplication by $f\notin \fp R_{\fm}$ on $N_{\fm}$ is injective. Since multiplication by $f\notin \fp R_{\fm}$ on $N_{\fm}$ is also surjective, $(N_{\fm})_{\fp}=N_{\fm}$. The rest of the proof follows the same line as in the previous case, but uses Proposition \ref{proposition: height d-2} instead. We will skip the details.\end{proof}

To summarize, under the assumption that $\mu^0(\fp,\HH^t_{\fp}(\scr{M}))\neq 0$, we have shown $\Ass_R(\HH^t_{\fp}(\scr{M}))=\{\fp\}$. Therefore, given any $f\notin \fp$, the multiplication by $f$ on $\HH^t_{\fp}(\scr{M})$ is injective. Since the multiplication by $f$ on $\HH^t_{\fp}(\scr{M})$ is also surjective ($\HH^t_{\fp}(\scr{M})$ is a quotient of an injective $R$-module), we have $\HH^t_{\fp}(\scr{M})\cong \HH^t_{\fp}(\scr{M})_{\fp}$ which is an injective $R_{\fp}$-module and hence isomorphic to a direct sum of copies of $\EE(R/\fp)$, which is {\it not} $F_R$-finite by Proposition \ref{proposition: height d-1}. This produces the desired contradiction since $\HH^t_{\fp}(\scr{M})$ is $F_R$-finite. 
\end{proof}

\begin{remark}
Following the same line as the proof of Theorem \ref{theorem: only max ideal at last spot}, one can prove the following: {\it let $R$ be a $d$-dimensional noetherian regular ring of prime characteristic and $\scr{M}$ be an $F_R$-finite $F_R$-module. If $\fp$ is a prime ideal of $R$ of height at most $d-2$ and set $t=\injdim_R(\scr{M})$, then $\mu^t(\fp,\scr{M})=0$}. 
\end{remark}

\begin{theorem}
\label{theorem: injective dim of F-module in char p}
Let $R$ be a regular ring of finite type over an infinite field $k$ of characteristic $p>0$. Then 
\[\injdim_R(\scr{M})=\dim_R(\Supp_R(\scr{M}))\]
for each $F_R$-finite $F_R$-module $\scr{M}$.
\end{theorem}
\begin{proof} First, we note that $R$ is a Jacobson ring (Proposition \ref{proposition: properties of Jacobson rings}). Hence Theorem \ref{theorem: only max ideal at last spot} is applicable.

We will use induction on $s=\dim_R(\Supp_R(\scr{M}))$. When $s=0$, the conclusion is clear.

Assume $s\geq 1$. Since $\scr{M}$ is $F_R$-finite, it has finitely many associated primes. Let $\fq_1,\dots,\fq_m$ be all the associated primes of $\scr{M}$ with $\dim(R/\fq_i)=s$. Since $k$ is infinite, by Noether normalization (\cite[Theorem 13.3]{EisenbudBookCommutativeAlgebra}), there are $x_1,\dots,x_d\in R$ that are algebraically independent over $k$ (where $d=\dim(R)$) so that $R$ is a finite $k[x_1,\dots,x_d]$-module and a linear combination of $x_1,\dots, x_d$, denoted by $y$, such that $k[y]\cap \fq_i=0$ for $i=1,\dots,m$. Set $S=k[y]\backslash \{0\}$. Then $S$ is a multiplicatively closed subset of $R$. Consider $S^{-1}R$, which is the same as $k(y)\otimes_{k[y]}R$. Note that $S^{-1}\scr{M}$ is also $F_{S^{-1}R}$-finite, and $S^{-1}R$ is of finite type over an infinite field $k(y)$. By Proposition \ref{proposition: properties of Jacobson rings}, $S^{-1}R$ is still a Jacobson ring. Also note that $\dim(S^{-1}R)=d-1$.

It is clear that $\dim_{S^{-1}R}(\Supp(_{S^{-1}R}(S^{-1}\scr{M})))=s-1$. Hence by our induction hypothesis 
\[\injdim_{S^{-1}R}(S^{-1}\scr{M})=s-1.\] 
Hence there exists a prime ideal $P$ in $S^{-1}R$ such that $\mu^{s-1}_{S^{-1}R}(P, \scr{M})\neq 0$. Let $\fp$ be the prime ideal in $R$ such that $\fp S^{-1}R=P$. Then, $\mu^{s-1}_{R}(\fp, \scr{M})\neq 0$. This already shows that $\injdim_R(\scr{M})\geq s-1$. With $\fp$ being a prime ideal in $S^{-1}R$, it follows that $\hgt(\fp)\leq d-1$. Theorem \ref{theorem: only max ideal at last spot} implies that $\injdim_R(\scr{M})\neq s-1$. Therefore $\injdim_R(\scr{M})=s$. This finishes the proof. 
\end{proof}

\begin{remark}
\label{remark: main result fail in local case}
Both Theorems \ref{theorem: only max ideal at last spot} and \ref{theorem: injective dim of F-module in char p} would fail if $R$ admitted a height $d-1$ prime ideal $\fp$ that's contained in only finitely many maximal ideals of $R$. Indeed, by Proposition \ref{proposition: height d-1}, $\EE(R/\fp)$ would be $F_R$-finite. It would be an injective $R$-module with a 1-dimensional support. 
\end{remark}


\section{Injective dimension of holonomic $D$-modules}
Throughout this section $R=k[x_1,\dots,x_n]$ denotes a polynomial ring over a field $k$. The ring of $k$-linear differential operators on $R$, denoted by $D(R,k)$, can be described explicitly as follows. Let $\partial^{[t]}_i$ denote the $k$-linear differential operators $\frac{1}{t!}\frac{\partial^t}{\partial x^t_i}$. Then $D(R,k)=R\langle \partial^{t_1}_1\cdots \partial^{t_n}_n\mid t_1,\dots,t_n\geq 0\rangle$.

\begin{proposition}
\label{prop: injective res is D-linear}
The minimal injective resolution of $R$
\[0\to R\to \II^0\xrightarrow{\delta^0} \cdots\to \II^j\xrightarrow{\delta^j} \cdots \to \II^n\xrightarrow{\delta^n} 0\]
where $\II^j\cong \bigoplus_{\hgt(\fp)=j}\EE(R/\fp)$, is an exact sequence in the category of $D(R,k)$-modules. Equivalently, each module in this resolution is a $D(R,k)$-module and each differential is $D(R,k)$-linear.
\end{proposition}
\begin{proof}
Since $R$ is regular and hence Gorenstein, \cite[Theorem 5.4]{SharpCousinComplex} shows that 
\[\II^j\cong \bigoplus_{\hgt(\fp)=j}\coker(\delta^{j-2})_{\fp}\] 
and $\delta^{j-1}$ is the composition of $\II^{j-1}\to \coker(\delta^{j-2})\to \II^j\cong \bigoplus_{\hgt(\fp)=j}\coker(\delta^{j-2})_{\fp}$. We will use induction on $j$ to show that each $\II^j$ is a $D(R,k)$-module and each $\delta^{j-1}$ is $D(R,k)$-linear. It is clear that $\II^0$ is the fractional field of $R$ and hence a natural $D(R,k)$-module. The natural inclusion $R\to \II^0$ is clearly $D(R,k)$-linear. Hence $\II^0/R$ is also a $D(R,k)$-module and so is $\II^1\cong \bigoplus_{\hgt(\fp)=1}(\II^0/R)_{\fp}$. Since $\II^0/R\to \II^1\cong \bigoplus_{\hgt(\fp)=1}(\II^0/R)_{\fp}$ is just the natural map from a $D(R,k)$ to a localization of it, it is $D(R,k)$-linear. Assume that we have proved our statement for $\II^l$ and $\delta^{l-1}$ with $l<j$. Then since $\II^j\cong \bigoplus_{\hgt(\fp)=j}\coker(\delta^{j-2})_{\fp}$ and $\delta^{j-1}$ is the composition of $\II^{j-1}\to \coker(\delta^{j-2})\to \II^j\cong \bigoplus_{\hgt(\fp)=j}\coker(\delta^{j-2})_{\fp}$, we see that $\II^j$ is a $D(R,k)$-module and $\delta^{j-1}$ is $D(R,k)$-linear. By induction, we have proved that all $\II^j$ are $D(R,k)$-modules and all $\delta^{j-1}$ are $D(R,k)$-linear for $0\leq j\leq n$. It remains to check $\delta^n$. But it is the zero map, clearly $D(R,k)$-linear. This finishes the proof of our proposition.
\end{proof}

\begin{proposition}
\label{proposition: height n-1 for D-modules}
Let $\fp$ be a prime ideal of $R$ with height $n-1$. Then $\EE(R/\fp)$ does not have finite length in the category of $D(R,k)$-module, and hence it is not holonomic.
\end{proposition}
\begin{proof}
The proof is nearly identical to the one of Proposition \ref{proposition: height d-1}; the only modification is to use finite length, instead of $F_R$-finiteness, to guarantee finiteness of associated primes. We skip the details.
\end{proof}

\begin{proposition}
\label{proposition: height n-2 for D-modules}
Let $\fp$ be a prime ideal of $R$ of height $\leq n-2$. Then $\EE(R/\fp)=\EE(R/\fp)_{\fm}$ does not have finite length in the category of $D(R_{\fm},k)$-modules for each maximal ideal $\fm$ that contains $\fp$.

In particular, if $\hgt(\fp)\leq n-2$, then $\EE(R/\fp)$ does not have finite length in the category of $D(R,k)$-modules.
\end{proposition}
\begin{proof}
First, by Proposition \ref{proposition: localization preserve finite length}, if $\EE(R/\fp)$ had finite length in the category of $D(R,k)$-modules, then so would $\EE(R/\fp)=\EE(R/\fp)_{\fm}$  in the category of $D(R_{\fm},k)$-modules. hence it suffices to prove the first conclusion. The proof of our first conclusion is nearly identical to the proof of Proposition \ref{proposition: height d-2}; the only modification is to use finite length, instead of $F_R$-finiteness, to guarantee finiteness of associated primes. We skip the details.
\end{proof}

The proof of the following theorem is a slight modification of the one of Theorem \ref{theorem: only max ideal at last spot}. For clarity and completeness, we include a proof.
\begin{theorem}
\label{theorem: only max ideal at last spot for D-modules}
Let $M$ be a holonomic $D(R,k)$-module. Set $\injdim_R(M)=t$. Then $\mu^t(\fp,M)=0$ for each non-maximal prime ideal $\fp$. 
\end{theorem}
\begin{proof}
According to \cite[Lemma 1.4]{LyubeznikFinitenessLocalCohomology}, $\mu^t(\fp,M)=\mu^0(\fp,\HH^t_{\fp}(M))$. Assume that $\mu^0(\fp,\HH^t_{\fp}(M))\neq 0$ and we will look for a contradiction. 

Since $\mu^0(\fp,\HH^t_{\fp}(M))\neq 0$, we must have $\HH^t_{\fp}(M)_{\fp}\neq 0$; consequently, $\fp$ (being the unique minimal element in the support of $\HH^t_{\fp}(M)$) must be an associated prime of $\HH^t_{\fp}(M)$. Under our assumption on $M$, we have that $\HH^t_{\fp}(M)$ has only finitely many associated primes. \vspace{3mm}

{\it Claim.} $\Ass_R(\HH^t_{\fp}(M))=\{\fp\}$.
\begin{proof}[Proof of Claim]
Assume otherwise. Let $\fp,\fq_1,\dots,\fq_m$ be the associated primes of $H^t_{\fp}(M)$, and set $J=\fq_1\cdots \fq_m$. Then $L:=\HH^0_J(\HH^t_{\fp}(M))$ is also $F$-finite. Let $N=\HH^t_{\fp}(M)/L$. We will show that $N=0$, which will produce the desired contradiction since $\fp$ is {\it not} an associated prime of $L$.

Since $t=\injdim(M)$, it follows that $\HH^t_{\fp}(M)$ is a quotient of an injective $R$-module. Given any element $f\in R$, the multiplication by $f$ on any injective module is surjective, hence it also surjective on $N$ and any localization of $N$.

If $\hgt(\fp)=n-1$, then $\fq_1,\dots,\fq_m$ are maximal ideals. Hence $L$ is an injective $R$-module, hence $\HH^t_{\fp}(M)=L\oplus N$. Since $N$ is a submodule of $\HH^t_{\fp}(M)$, each associated prime must be an associated prime of $\HH^t_{\fp}(M)$. It is clear that none of $\fq_1,\dots,\fq_m$ is an associated prime of $N$. Therefore $\fp$ is the only associated prime of $N$. Consequently multiplication by $f\notin \fp$ is also injective on $N$. Thus, $N=N_{\fp}$. Since $\fp$ is a minimal prime of $N$, by Proposition \ref{prop: D-module localized at minimal prime}, $N_{\fp}$ is an injective $R_{\fp}$-module. Since $\fp$ is the only associated prime of $N$, it follows that $N_{\fp}$ is a direct sum of finitely copies of $\EE(R_{\fp}/\fp R_{\fp})=\EE(R/\fp)$. To summarize, we have shown that $N$, which is holonomic, is a direct sum of finitely many copies of $\EE(R/\fp)$. Since $R$ is a Jacobson ring, so is $R/\fp$ (Proposition \ref{proposition: properties of Jacobson rings}). Hence there are infinitely many maximal ideals that contain $\fp$. By Proposition \ref{proposition: height n-1 for D-modules}, $\EE(R/\fp)$ is {\it not} holonomic; thus $N$ must be 0.

Assume now $\hgt(\fp)\leq n-2$. Since $R$ is a Jacobson ring. Hence there exists a maximal ideal $\fm$ that contains $\fp$ but not any of $\fq_1,\dots,\fq_m$ (Remark \ref{remark: existence of maximal ideal}). Hence $N_{\fm}=\HH^t_{\fp}(M)_{\fm}$. Over $R_{\fm}$, the only associated prime of $\HH^t_{\fp}(M)_{\fm}=N_{\fm}$ is $\fp R_{\fm}$. Consequently multiplication by $f\notin \fp R_{\fm}$ on $N_{\fm}$ is injective. Since multiplication by $f\notin \fp R_{\fm}$ on $N_{\fm}$ is also surjective, $(N_{\fm})_{\fp}=N_{\fm}$. As in the previous paragraph, $(N_{\fm})_{\fp}$ is an injective $R_{\fp}$-module. Since $\fp$ is a minimal prime of $N$, it follows that $(N_{\fm})_{\fp}$ is a direct sum of $\EE(R/\fp)$. By Proposition \ref{proposition: localization preserve finite length}, $(N_{\fm})_{\fp}=N_{\fm}$ has finite length in the category of $D(R_{\fm},k)$-modules. If $(N_{\fm})_{\fp}$ were not zero, then $\EE(R/\fp)$ would have finite length in the category of $D(R_{\fm},k)$-modules, contradicting Proposition \ref{proposition: height n-2 for D-modules}. So $N_{\fp}=(N_{\fm})_{\fp}=0$. But $\fp$ is in the support of $N$, this forces $N=0$.\end{proof}

To summarize, under the assumption that $\mu^0(\fp,\HH^t_{\fp}(M))\neq 0$, we have shown $\Ass_R(\HH^t_{\fp}(M))=\{\fp\}$. Therefore, given any $f\notin \fp$, the multiplication by $f$ on $\HH^t_{\fp}(M)$ is injective. Since the multiplication by $f$ on $\HH^t_{\fp}(M)$ is also surjective ($\HH^t_{\fp}(M)$ is a quotient of an injective $R$-module), we have $\HH^t_{\fp}(M)\cong \HH^t_{\fp}(M)_{\fp}$ which is an injective $R_{\fp}$-module and hence isomorphic to a direct sum of copies of $\EE(R/\fp)$, which is {\it not} holonomic by Proposition \ref{proposition: height n-1 for D-modules}. This produces the desired contradiction since $\HH^t_{\fp}(M)$ is holonomic. 
\end{proof}

\begin{theorem}
\label{thm: injective dim of holonomic D-module}
Let $k$ be a field of characteristic 0 and $R=k[x_1,\dots,x_n]$ be a polynomial ring over $k$. If $M$ is a holonomic $D(R,k)$-module, then 
\[\injdim_R(M)=\dim_R(\Supp_R(M)).\]
\end{theorem}
\begin{proof}
The proof follows the same line as in the one of Theorem\ref{theorem: injective dim of F-module in char p}; we opt to include a proof here for the sake of clarity and completeness. We will use induction on $s=\dim_R(\Supp_R(M))$. When $s=0$, the conclusion is clear by Proposition \ref{prop: injdim for D-module}.

Assume $s\geq 1$. Since $M$ is holonomic, it has finitely many associated primes. Let $\fq_1,\dots,\fq_m$ be all the associated primes of $M$ with $R/\fq_i=s$. Since $k$ is infinite, by Noether normalization (\cite[Theorem 13.3]{EisenbudBookCommutativeAlgebra}), there are $x_1,\dots,x_d\in R$ that are algebraically independent over $k$ (where $d=\dim(R)$) so that $R$ is a finite $k[x_1,\dots,x_d]$-module and a linear combination of $x_1,\dots, x_d$, denoted by $y$, such that $k[y]\cap \fq_i=0$ for $i=1,\dots,m$.  Set $S=k[x_n]\backslash \{0\}$. Then $S$ is a multiplicatively closed subset of $R$. Note that $S^{-1}R=k(x_n)[x_1,\dots,x_{n-1}]$. It follows from Proposition \ref{prop: holonomic localized} that $S^{-1}M$ is a holonomic $D(S^{-1}R,k(x_n))$-module. It is clear that $\dim_{S^{-1}R}(\Supp(_{S^{-1}R}(S^{-1}M)))=s-1$. Hence by our induction hypothesis 
\[\injdim_{S^{-1}R}(S^{-1}M)=s-1.\] 
Hence there exists a prime ideal $P$ in $S^{-1}R$ such that $\mu^{s-1}_{S^{-1}R}(P, M)\neq 0$. Let $\fp$ be the prime ideal in $R$ such that $\fp S^{-1}R=P$. Then, $\mu^{s-1}_{R}(\fp, M)\neq 0$. This already shows that $\injdim_R(M)\geq s-1$. With $\fp$ being a prime ideal in $S^{-1}R$, it follows that $\hgt(\fp)\leq d-1$. Theorem \ref{theorem: only max ideal at last spot for D-modules} implies that $\injdim_R(M)\neq s-1$. Therefore $\injdim_R(M)=s$. This finishes the proof. 
\end{proof}

\begin{remark}
According to \cite[2.2(d)]{LyubeznikFinitenessLocalCohomology}, $\mathcal{T}(R)$ is a holonomic $D(R,k)$-module for each Lyubeznik functor $\mathcal{T}$. Therefore, our Theorems \ref{theorem: only max ideal at last spot for D-modules} and \ref{thm: injective dim of holonomic D-module} generalize the main results in \cite{PuthenpurakaInjectiveResolutionofLC}.
\end{remark}

\bibliographystyle{skalpha}
\bibliography{injectivedimension}

\end{document}